\documentclass{amsart}

\usepackage{booktabs,siunitx}
\usepackage{mathtools}
\usepackage[parfill]{parskip}
\usepackage{amscd,amssymb}
\usepackage{verbatim}
\usepackage[T2A]{fontenc}
\usepackage{array,calc}
\usepackage{eqnarray,amsmath}
\usepackage{tikz}
\usepackage{graphicx}
\usepackage{booktabs,tabularx}
\usepackage{rotating}
\usepackage{authblk}
\usepackage{array,makecell,multirow}
\makeatletter
\def\subsubsection{\@startsection{subsubsection}{3}%
  \z@{.5\linespacing\@plus.7\linespacing}{.1\linespacing}%
  {\normalfont\itshape}}
\makeatother

\graphicspath{ {latex uchene/} }
\usepackage{amsmath}
\newtheorem{theorem}{Theorem}[section]

\newtheorem{lemma}[theorem]{Lemma}
\newtheorem{proposition}[theorem]{Proposition}
\newtheorem{corollary}[theorem]{Corollary}

\theoremstyle{definition}
\newtheorem{definition}[theorem]{Definition}

\email{chavdar.lalov@gmail.com}
\title{Self-avoiding walks on lattice strips}

\begin{document}

\maketitle
\begin{center}
R. Dangovski$^1$ ~~~ \and ~ C. Lalov$^2$ 
\end{center}
\begin{center}
1. Massachusetts Institute of Technology, Cambridge, USA\\
2. Geo Milev High School of Mathematics, Pleven, Bulgaria 
\end{center}
\begin{abstract}
We study self-avoiding walks on restricted square lattices, more precisely on the lattice strips $\mathbb{Z} \times \{-1,0,1\}$ and $\mathbb{Z}\times \{-1,0,1,2\}$. We obtain the value of the connective constant for the $\mathbb{Z} \times \{-1,0,1\}$ lattice in a new shorter way and deduce close bounds for the connective constant for the $\mathbb{Z}\times \{-1,0,1,2\}$ lattice. Moreover, for both lattice strips we find close lower and upper bounds for the number of SAWs of length $n$ by using the connective constant. 
\end{abstract}

\section{Introduction}
A self-avoiding walk (SAW) is a path that does not self-intersect, i.e. it does not pass through a given point more than once. SAWs were introduced by Flory, a theoretical chemists, in order to model the behaviour of linear polymers in dilute solutions ([F]  p. 672.). They quickly became an intriguing combinatorial problem for mathematicians and an interesting computational problem for computer scientists (see [S]).  
~\\~\\
In this paper we study SAWs on the two-dimensional square lattice.

\begin{definition}
A self-avoiding walk (SAW) of length $n$ is a sequence of points  $(w_0,w_1,\dots,w_n)$ where $w_0=(0,0)$ and $w_i=(x_i,y_i)$ ($x_i$ and $y_i$ integers), such that no point repeats itself and $|x_{i+1}-x_i|+|y_{i+1}-y_i|=1$ for all $i$, where $i\in \{0,1,\dots,n-1 \}$. We define the move from point $w_i$ to point $w_{i+1}$  to be called a \emph{step}.
\end{definition}

We investigate the number of SAWs on the lattice strips $\mathbb{Z} \times \{-1,0,1\}$ and $\mathbb{Z}\times \{-1,0,1,2\}$. For the lattice $\mathbb{Z}\times \{x,x+1,\dots, y\}$ we denote the number of SAWs of length $n$ with $c_{[x,y]n}$. If $x$ and $y$ are respectively minus and plus infinity then we just write $c_n$.

An exact formula for $c_{[x,y]n} $ has only been obtained for $c_{[0,1]n}$ ([Z], [B], [N]).
\begin{theorem}[Zeilberger 1996]
 For $n\geq2$, the number of SAWs of length $n$ on the grid $\mathbb{Z} \times\{0,1\}$ is \setlength{\abovedisplayskip}{0pt}
$$c_{[0,1]n}=8F_n-\delta_n$$
where $F_n$ is the $n^{th}$ Fibonacci number and $\delta_n=4$ if $n$ is odd and $\delta_n=n$ otherwise. 
\end{theorem}
It is conjectured $c_n \sim A\mu^n n^v$, 
where $A$, $\mu$ and $v$ are constants. The most studied one of them is the \emph{connective constant} ($\mu$).

\begin{definition}
 The limit $\lim_{n \rightarrow \infty}\sqrt[n]{c_n}$ is called the \textit{connective constant} and we denote it by $\mu_{\mathbb{Z}\times\mathbb{Z}}$. For the lattice strip $\mathbb{Z}\times\{x,x+1,\dots,y\} $ we define the connective constant analogously and denote it by $\mu_{[x,y]}$
\end{definition}

The connective constant could be proven to exist and to be finite by using Fekete's lemma and that $c_{n}c_{m}\geq c_{n+m}$. There are also strong bounds for $c_{n}$ using $\mu_{\mathbb{Z} \times \mathbb{Z}}$ in [MS]. However, even the value for $\mu_{\mathbb{Z} \times \mathbb{Z}}$ is not known.

A way one could obtain $\mu$ is to observe when does the series of the generating function (for reading on generating functions see [W1])
$$G(t)_{[x,y]}=1+c_{[x,y]1}t+c_{[x,y]2}t^2+…$$
converge. $G(t)_{[x,y]}$ has been found only when $x=0$ and $y=1$ (see [Z]) and when $x=-1$ and $y=1$ (see [D]). Respectively, $\mu_{[0,1]}$ equals the golden ratio and $\mu_{[-1,1]}\approx 1.914\dots$.

A major breakthrough was made by H. Kesten. He showed that there exists a class of walks called bridges which have the same connective constant as all SAWs on the square lattice (see [K]).

\begin{definition} Let us have a self-avoiding walk with coordinates $(w_0,w_1,\dots,w_n)$, as $w_0=(0,0)$. Then a SAW is a bridge if  $x_0 < x_j \leq x_n$ for all $j>0$, where $x_i$ is the $x$-axis coordinate of the $i^{th}$ point. We use $b_n$ to denote the number of $n$-step bridges on the $\mathbb{Z} \times \mathbb{Z}$ grid. Kesten showed that $b_n^\frac{1}{n}$ converges to $\mu_{\mathbb{Z} \times \mathbb{Z}}$ as $n {\rightarrow} \infty $. For the lattice strip $\mathbb{Z}\times\{x,x+1,\dots,y\}$ we denote the number of bridges by $b_{[x,y]n}$.
\end{definition}

We show Kesten's result is true for the $\mathbb{Z} \times \{-1,0,1\}$ and $\mathbb{Z}\times \{-1,0,1,2\}$ lattices  as well. We do that by using the Hammersley-Welsh method (see [MS] p. 57). 

In this paper we use bridges in order to find $\mu_{[-1,1]}$ in a new shorter way, decreasing the length of the proof dramatically, and to obtain close lower and upper bounds for $\mu_{[-1,2]}$. Our method yields that $\frac{1}{\mu_{[-1,1]}}$ is the smallest modulus root of the equation $1-t-2t^3-t^4-2t^5-2t^6=0$ and that $2.050\leq\mu_{[-1,2]}\leq2.166$.  Moreover, the Hammersley-Welsh method allows us to derive close lower and upper bound for the number of SAWs of lenght $n$ by using $\mu$. In the end we do not have the exact number of walks but we do have a good idea for their growth rate.

\section{ The connective constant on the $\mathbb{Z} \times \{-1,0,1\}$ lattice}

We prove the following result 
\begin{equation}
    \mu_{[-1,1]}^n \leq c_{[-1,1]n}\leq \mu_{[-1,1]}^{n+1}((n+1)+2(n+1)^2+3(n+1)^3+2(n+1)^4+(n+1)^5)\end{equation}
for $n\geq1$ in appendix B by using the Hammersley-Welsh Method. From the proof one gets as a bonus that $\lim_{n\rightarrow\infty}\sqrt[n]{b_{[-1,1]n}}=\mu_{[-1,1]}$.

Therefore, the problem translates to counting bridges. To do that we use Zeilberger's Decomposition Method (see [Z]) and \emph{irreducible bridges}. Zeilberger's Decomposition Method is an ``alphabetical'' approach for describing the self-avoiding walk. A generalisation of this method is presented in [W2] in order to count up-side SAWs. One divides the SAWs in their basic movements which they perform and for each of them one chooses a symbol to represent it. These symbols become our ``letters'' of an ``alphabet'' for creating different ``words'' (SAWs).  For example, our language for the  $\mathbb{Z} \times \{-1,0,1\}$ grid consists of 3 different two letter combinations allowing us to construct every bridge.  
\begin{itemize}
\item $I_O$ --- denotes an irreducible bridge between an inner and outer line;
\item $O_O$ --- denotes an irreducible bridge between the two outer lines;
\item $O_I$ --- denotes an irreducible bridge between an outer and inner line.
\end{itemize}
The two outer lines are the horizontal lines $y=1$ and $y=-1$. The inner line is the horizontal line $y=0$.

We proceed with the explanation of what is an \textbf{irreducible bridge}. Whenever we join two bridges of lengths $m$ and $n$ we get another bridge of length $m + n$. As a result, every bridge can be decomposed into \textit{irreducible bridges} (bridges which cannot be decomposed further). However, now a straight line of $k$ right steps is actually $k$ irreducible bridges of length $1$ glued to each other. Therefore, for our convenience let us make the following correction in our understanding of what an irreducible bridge will be (this new understanding is used above, when describing the two-letter combinations $I_O$, $O_O$ and $O_I$). If we have a line of $k$ irreducible bridges of length $1$ and then another irreducible bridge of length bigger than $1$, then we will consider all these $k+1$ irreducible bridges as a single irreducible bridge (see fig. 1). We will relate to the $k$ right steps in the beginning of the irreducible bridge as its \textbf{tail}.

\begin{figure}[h!]

\begin{tikzpicture}[scale=0.8]
\draw[step=1cm,gray,very thin] (-2,-2) grid (6,0);

\draw[blue,ultra thick] (-2,-2) -- (6,-2) -- (6,-1) -- (1,-1) -- (1,0) -- (6,0);
\end{tikzpicture}
\caption{An $O_O$ irreducible bridge with a tale of length $2$.}
\end{figure}
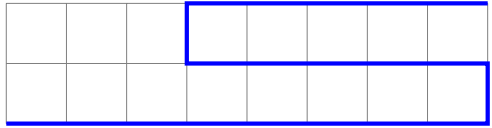
\begin{proposition}The following code using $O_O$, $O_I$ and $I_O$ describes all bridges. 
$$[I_O {O_O}^* O_I]^*\widetilde{I_O{ O_O}^*} r^*$$ 
where the symbol '*' is an indefinite superscript $(\geq0)$, the tilde indicates that the movement may not be executed and $r$ is a step to the right.
\end{proposition}
The tilde and star can be considered as the ``punctuation'' needed for our ``language''.

In order to deduce the generating function of all bridges we need to find the generating function for each of the  types of irreducible bridges --- $O_O$, $O_I$ and $I_O$.

In this section we will obtain the generating function of the most complex one, the $O_O$ type. Firstly, we derive the generating function of the irreducible bridges without a tail. Let us assume we are on the higher outer line (because the grid is symmetric it does not matter). The irreducible bridges follow a certain pattern. For $k\geq0$ they always consists of $k+1$ steps to the right followed by a step down and $k$ steps to the left. Then we have another step down and $k$ steps to the right. The generating function is:
$$\frac{t^3}{1-t^3}$$
We can add the tail by multiplying the generating function by $\frac{1}{1-t}$.
$$G_{OO}=\frac{t^3}{1-t-t^3+t^4}$$

We find the generating functions of the other types of irreducible bridges similarly (Fig. 2).
\begin{center}
\def\arraystretch{1.5}
\begin{tabular}{|c|c|}\hline
Irred. Bridge & Generating function\\ \hline
    $O_I$ & $\frac{t^2}{(1-t)}=t^2+t^3+t^4+t^5+\dots$  \\[3pt ] \hline
     $I_O$ & $\frac{2t^2}{(1-t)}=2t^2+2t^3+2t^4+2t^5+\dots$\\[3pt]\hline
     $O_O$ & $\frac{t^3}{1-t-t^3+t^4}=t^3 + t^4 + t^5 + 2 t^6+\dots$ \\[3pt] 
     \hline
\end{tabular} 
\end{center}
\begin{center}
    Table 1.
\end{center}
We need to calculate the generating function  of ${O_O}^*$. ${O_O}*$ is just an indefinite number of $O_O$ irreducible bridges glued to each other. Therefore, the generating function is 
$$
G(t)_{OO*}= \frac{1}{1-\frac{t^3}{1-t-t^3+t^4}}=\frac{1-t-t^3+t^4}{1-t-2t^3+t^4}.$$
Hence, $$G(t)_{[-1,1]}= \frac{1}{1-G(t)_{IO}G(t)_{OO*}G(t)_{OI}}(1+G(t)_{IO}G(t)_{OO*})G(t)_{u*}=$$

$$=
\tfrac{1 - 5 t + 12 t^2 - 22 t^3 + 35 t^4 - 47 t^5 + 56 t^6 - 58 t^7 + 
   49 t^8 - 37 t^9 + 25 t^{10} - 11 t^{11} + 2 t^{12}}{1 - 6 t + 15 t^2 - 
   24 t^3 + 35 t^4 - 48 t^5 + 53 t^6 - 46 t^7 + 31 t^8 - 16 t^9 + 
   4 t^{10} + 10 t^{11} - 17 t^{12} + 10 t^{13} - 2 t^{14}}.$$
Hence, we have obtained our desired generating function. In order to deduce the value of $\mu_{[-1,1]}$, we find the smallest modulus root of the denominator of $G(t)_{[-1,1]}$, and take its reciprocal:
\begin{proposition}The value of the connective constant on the $\mathbb{Z}\times \{-1,0,1\}$ lattice is approximately:
 $$ \mu_{[-1,1]} \approx \frac{1}{0.522295}\approx 1.914.$$
 \end{proposition}
However, if we look carefully, the real structure of the SAW is encoded in the generating function of $[I_O {O_O}^* O_I]^*$. One could only calculate that generating function and it turns out that the desired root is really there. The denominator reduces to the simple polynomial $1-t-2t^3-t^4-2t^5-2t^6$.

\section{Self-avoiding walks on the $\mathbb{Z} \times \{-1,0,1,2\}$ grid}

In appendix B we prove that \begin{equation}\mu_{[-1,2]}^n \leq c_{[-1,2]n}\leq \mu_{[-1,2]}^{n+1}P(n)\end{equation}
where $P(n)=(n+1)+2(n+1)^2+3(n+1)^3+4(n+1)^4+3(n+1)^5+2(n+1)^6+(n+1)^7$ for $n\geq1$
by using the Hammersley-Welsh Method. It yields as a bonus that $\lim_{n\rightarrow\infty}\sqrt[n]{b_{[-1,2]n}}=\mu_{[-1,2]}$.

We again introduce a similar linguistic approach to the problem and use the same definition for an irreducible bridge. We use the same irreducible bridges of types $O_O$, $O_I$, $I_O$, however, we do introduce one more. $I_I$ will denote an irreducible bridge between the two inner lines, because now we have two outer ($y=2$ and $y=-1$) and two inner ($y=0$ and $y=1$) lines.

\begin{proposition}
Every bridge can be encoded in the following way:
$$[{I_I}^* I_O {O_O}^* O_I]^* {I_I}^* \widetilde{I_O{ O_O}^*} r^*.$$
\end{proposition}
Hence, the problem translates to finding the generating function for each of the four types of irreducible bridges. However, on the $\mathbb{Z} \times \{-1,0,1,2\}$ lattice, SAWs have much more freedom and thus, their behaviour is harder to describe. If we want to count all bridges, we will have to count SAWs, such as the one shown in fig. 2, and in some way include it in our generating function. The calculations soon becomes quite unpleasant. Nevertheless, we can use our method to derive close bounds for the connective constant. For example, if we count parts of the irreducible bridges of each type ($O_O$, $I_O$, \dots) and analyse the resulting generating function, we will obtain a lower bound for the connective constant. The same reasoning is used for for the upper bound.

\begin{figure}[h!]
\begin{tikzpicture}[scale=0.7]
\draw[step=1cm,gray,very thin] (-2,-2) grid (10,1);
\draw[blue,ultra thick] (-2,1) -- (5,1) -- (5,0) -- (6,0) -- (8,0) -- (8,1) -- (10,1) -- (10,-1) -- (4,-1) -- (4,0) -- (1,0) -- (1,-1) -- (1,-2) -- (2,-2) -- (2,-1) -- (3,-1) -- (3,-2) -- (10,-2) ; 
 
\end{tikzpicture}

\caption{An $O_O$ irreducible bridge with a tail of length $2$. }
\end{figure}
We show examples of how we can obtain ``good''
bounds for the connective constant, although if one makes more calculations, even better bounds could be found.
\subsection{Lower bound for the connective constant}
For our lower bound we need to count part of the irreducible bridges for each of the types $O_O$, $O_I$, $I_I$ and $I_O$. Let us count only those which do not make a step to the left. Hence, we will count right-side walks, whose first step is to the right. We present a table with the generating functions:
\begin{center}
\def\arraystretch{1.5}

\begin{tabular}{|c|c|}\hline
Irred. Bridge & Generating function\\ \hline

    $I_I$ &  $\frac{t^2}{1-t}=t^2+t^3+t^4+t^5\dots$ \\[3pt] \hline
    $O_I$ & $\frac{t^2+t^3}{(1-t)}=t^2+2t^3+2t^4+2t^5+\dots$ \\[3pt] \hline
     $I_O$ & $\frac{t^2+t^3}{(1-t)}=t^2+2t^3+2t^4+2t^5+\dots$ \\[3pt] \hline
     $O_O$ & $\frac{t^4}{(1-t)}=t^4 + t^5 + t^6 + t^7+\dots$ \\[3pt] 
     \hline
\end{tabular}
\end{center}
\begin{center}
Table 2.
\end{center}

Therefore, the generating function of all right-side SAWs, which start with a right step is:
$$\frac{1 - t + t^2 + t^3 - t^4}{1 - 2 t + t^3 - 2 t^4 - t^5}=1 + t + 3 t^2 + 6 t^3 + 12 t^4 + 24 t^5\dots $$.

The radius of convergence is approximately $0.487645$.
\begin{corollary}The following lower bound for the connective constant is true:
$$\frac{1}{0.487645} \approx 2.050\leq \mu_{[-1,2]}.$$
\end{corollary}
\subsection{Upper bound for the connective constant}
We proceed with the upper bound. Through exhaustive search we discover that there are 2 types of irreducible bridges, which follow certain patterns:

We may have a tail. Then  we have a step to the right and respectively one, two or three steps upward or downward depending on which type of irreducible bridge we are performing - we will call these irreducible bridges \textbf{simple} (fig. 3).

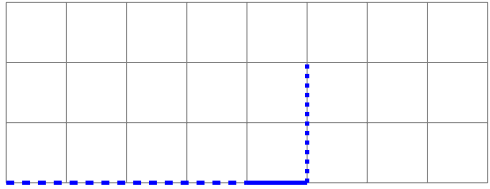
\begin{figure}[h!]
\begin{tikzpicture}[scale=0.8]
\draw[step=1cm,gray,very thin] (-2,-2) grid (6,1);
\draw[blue,ultra thick,dashed] (-2,-2) -- (2,-2);
\draw[blue,ultra thick] (2,-2) -- (3,-2);
 \draw[blue,ultra thick,dotted](3,-2) -- (3,0) ; 
 
\end{tikzpicture}

\caption{A simple $O_I$ irreducible bridge. We have a tail of length four followed by a step to the right. Afterwards we have two steps upwards.}
\end{figure}

The other pattern is as follows. We assume we are on one of the two lower lines, parallel to the $x$-axis ($y = 0$ and $y = −1$), as the other case is analogical. We may have a tail. Then we have 3 avoidable SAWs on strips with length 1, which go in ``opposite'' directions, as the first walk goes right (on the strip $\mathbb{Z} \times \{-1, 0\}$), the second one goes left (on the strip $\mathbb{Z} \times\{ 0,1\}$), and the third one goes to the right again (on the strip $\mathbb{Z} \times\{1,2\}$). Their starting and ending points are determined by the type of irreducible bridge. We call these bridges \textbf{complicated} (fig. 4).

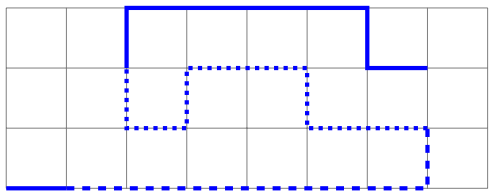
\begin{figure}[h!]
\begin{tikzpicture}[scale=0.8]
\draw[step=1cm,gray,very thin] (-2,-2) grid (6,1);
\draw[blue,ultra thick] (-2,-2) -- (-1,-2);
\draw[blue,ultra thick,dashed] (-1,-2) -- (2,-2) -- (5,-2) -- (5,-1);
 \draw[blue,ultra thick,dotted](5,-1) -- (3,-1) -- (3,0) -- (1,0) -- (1,-1) -- (0,-1) -- (0,0) ;
\draw[blue,ultra thick] (0,0) -- (0,1) -- (4,1) -- (4,0) -- (5,0); 
 
\end{tikzpicture}

\caption{A complicated $O_I$ irreducible bridge. We have a tail of length $1$. Afterwards we have a walk that goes right on the strip
$\mathbb{Z} \times \{-1, 0\}$ of length $7$. Then a walk that goes left on the strip $\mathbb{Z} \times\{0, 1\}$ of length $8$ followed by a walk that goes right on the strip $\mathbb{Z} \times\{1, 2\}$ of length $7$.
}
\end{figure}
First, we examine irreducible bridges of type $O_O$. 

Let us imagine a complicated bridge that starts from the higher outer line.  We ignore the tail for now and perform the following operation: do not move the walk that goes right on the grid $\mathbb{Z} \times\{1, 2\}$. However, we add a right step where it ends and place the beginning of the walk that goes left in the end of the added right step and rotate it about its beginning by $180$ degrees. It became a walk that goes right. We proceed by adding another right step and at its end we place the beginning of the walk that goes right on the grid $\mathbb{Z} \times\{-1, 0\}$. Hence, we transformed our irreducible bridge without a tail into a walk with two more steps on the grid $\mathbb{Z} \times\{1, 2\}$, as we know the lines on which it starts and ends. (We show an example on figures 5 and 6.) Notice that we add right steps because otherwise our transformed walk may not be a SAW. Moreover, every irreducible bridge $O_O$ without a tail has a unique transformed walk. However, the inverse is not true.

 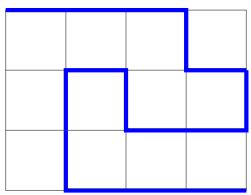
\begin{figure}[h!]
\begin{tikzpicture}[scale=0.8]
\draw[step=1cm,gray,very thin] (-2,-2) grid (2,1);
\draw[blue,ultra thick] (-2,1) -- (1,1) -- (1,0) -- (2,0);
\draw[blue,ultra thick] (2,0) -- (2,-1) -- (1,-1) -- (0,-1) -- (0,0) -- (-1,0) -- (-1,-1) ;
 \draw[blue,ultra thick] (-1,-1) -- (-1,-2) -- (2,-2);
\end{tikzpicture}

\caption{$O_O$ \textit{irreducible bridge} without a \textit{tail}.}
\end{figure}

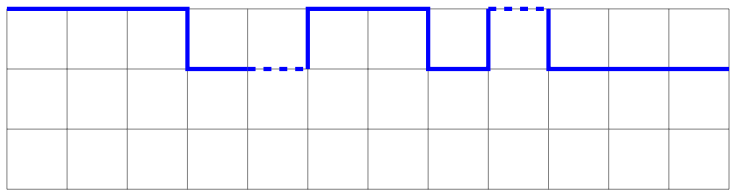
\begin{figure}[h!]
\begin{tikzpicture}[scale=0.8]
\draw[step=1cm,gray,very thin] (-2,-2) grid (10,1);
\draw[blue,ultra thick]  
(-2,1) -- (1,1) -- (1,0) -- (2,0);
\draw[blue,ultra thick,dashed](2,0) -- (3,0);
\draw[blue,ultra thick](3,0) -- (3,1) -- (5,1) -- (5,0) -- (6,0) -- (6,1);
\draw[blue,ultra thick,dashed](6,1) -- (7,1);
\draw[blue,ultra thick](7,1) -- (7,0) -- (10,0);
\end{tikzpicture}

\caption{The bridge from Fig. 5 after the transformation.}
\end{figure}

We now find the generating function of the transformed SAWs. Such a transformed SAW cannot make left steps, it is only on the grid $\mathbb{Z} \times$\{$1, 2$\} , its starting and ending points are on different lines (beginning on the outer line and ending on the inner one) and the first step is always to the right.

The generating function equals:

$$ \frac{\frac{t^2}{1-t}}{1-\frac{t^4}{(1-t)^2}} \frac{1}{1-t}=\frac{t^2}{1-2t+t^2-t^4}.$$

If we divide the above generating function by $t^2$ (because of the added 2 right steps) we will obtain a generating function whose coefficient in front of the $n^{th}$ power is bigger than or equal to the number of  irreducible bridges of type $O_O$ without a tail of length $n$.

$$\frac{1}{1-2t+t^2-t^4}=1 + 2 t + 3 t^2 + 4 t^3 + 6 t^4 + 10 t^5 \dots$$

However, as we want to make the coefficients more accurate, we will decrease the coefficients in front of the first 14 powers.  (We check by exhaustive search that these coefficients can be decreased.) 

\begin{small}
\begin{eqnarray*}
\frac{1}{1-2t+t^2-t^4} & - & (1 + 2 t + 3 t^2 + 4 t^3 + 5 t^4 + 10 t^5 + 17 t^6 + 24 t^7 + 
  45 t^8 + 72 t^9 + 
  \\
 & & {}+109 t^{10} +188 t^{11} +301 t^{12} + 474 t^{13})=\end{eqnarray*}\end{small}
  $$= \tfrac{t^4-2t^5 + t^6+4 t^7-9t^8+4t^9 + 7 t^{10} - 18 t^{11} + 
 11 t^{12} + 12 t^{13} + 756 t^{14} - 286 t^{15} + 301 t^{16} + 474 t^{17}}{1-2t+t^2-t^4}.$$

 Moreover, we notice that we also counted the simple irreducible bridges in the generating function. (We did not decrease the coefficient in front of the fourth power with $6$.) Now we can add the tail by multiplying by $\frac{1}{1-t}$. Hence we will obtain a generating function whose coefficient in front of the $n^{th}$ power is bigger than or equal to the number of irreducible bridges of type $O_O$ with length $n$.
 
 The other types of irreducible bridges will be considered in the appendix A.
 
 In the end, our calculations yield one ``big'' generating function. We again look at the ``important'' part of the code of all bridges:
  $$[I_OO_O^*O_II_I^*]^* $$
The denominator of the generating function of this part of the code is 
\begin{eqnarray*}
D(t) & = & 1 - 12 t + 65 t^2 - 209 t^3 + 434 t^4 - 568 t^5 + 338 t^6 + 305 t^7 - 
 907 t^8 + \\
 & & {}  +770 t^9+ 292 t^{10} - 1462 t^{11} + 1406 t^{12}+ 446 t^{13} - 
 3945 t^{14} + 13408 t^{15}-\\
 & & {}  - 42903 t^{16} + 101573 t^{17} - 158117 t^{18} + 
 136952 t^{19} + 4507 t^{20}-\\ & & {}- 182921 t^{21} + 225943 t^{22} - 49787 t^{23} - 
 215357 t^{24}
 + 317489 t^{25}- \\
 & & {} - 108470 t^{26} - 314100 t^{27} + 
 801774 t^{28}- 1620468 t^{29} + 3204285 t^{30}-\\
 & & {}  - 4939210 t^{31} + 
 4697564 t^{32}  - 1024682 t^{33} - 3939143 t^{34} + 5903640 t^{35}-\\
 & & {}  - 
 3220560 t^{36} - 980952 t^{37}+ 2685716 t^{38} - 1510904 t^{39} - 
 162295 t^{40}+ \\& & {}+ 605850 t^{41}- 239118 t^{42} - 42432 t^{43} + 55764 t^{44}.  \\
\end{eqnarray*}

 The radius of convergence is approximately $0.461722$. 
 \begin{corollary}
 The following upper bound for the connective constant is true:
 $$\mu_{[-1,2]}\leq 2.166.$$
 \end{corollary}
 \section{Acknowledgements}
The second author is partially supported  by the High School Student Institute at the Bulgarian Academy of Science and Club Young Scientists, Bulgaria.

\newpage

\appendix

\section{Upper bound for $\mu_{[-1,2]}$}
We consider the other 3 types of irreducible bridges. We use the same method to make an injective transformation that maps the irreducible bridges to SAWs on a lattice strip of width one. Then, after finding the generating function of the codomain, we decrease some of the coefficients. The calculations are in table 3.
\def\arraystretch{1.4}

\begin{tabular}{|c|c|c|c|c|}\hline
~&$O_O$ & $I_O$ & $O_I$& $I_I$\\ \hline
    Transformed SAWs&$\frac{t^2}{1-2t+t^2-t^4}$& $\frac{(1-t)t}{1-2t+t^2-t^4}$ &$\frac{(1-t)t}{1-2t+t^2-t^4}$&$\frac{t^2}{1 - 2 t + t^2 - t^4}$ \\[3pt ] 
     \hline
     Extracting the added steps&$\frac{1}{1-2t+t^2-t^4}$& $\frac{(1-t)}{1-2t+t^2-t^4}$ &$\frac{(1-t)}{1-2t+t^2-t^4}$&$\frac{t^2}{1 - 2 t + t^2 - t^4}$\\[3pt ] 
     \hline
\parbox[t]{2mm}{\multirow{21}{*}{\rotatebox[origin=c]{90}{Decreasing the  coefficients with \dots~  ~~~  ~}}}  &\parbox[t]{2mm}{\multirow{11}{*}{\rotatebox[origin=c]{90}{
\begin{small}
$1 + 2t + 3t^2 + 4t^3 + 5t^4 + 10t^5+17t^6+24t^7+45t^8+72t^9+109t^{10}+188t^{11}+301t^{12}+474t^{13}$
~~\end{small}}}} &
 \parbox[t]{2mm}{\multirow{11}{*}{\rotatebox[origin=c]{90}{$1 + t + 2 t^4 + 4 t^5 + 6 t^6 + 11 t^7 + 16 t^8 + 26 t^9 + 44 t^{10} + 
 67 t^{11} + 115 t^{12} + 180 t^{13}$~~} }} &
\parbox[t]{5mm}{\multirow{11}{*}{\rotatebox[origin=c]{90}{$1 + t + 2 t^4 + 4 t^5 + 6 t^6 + 11 t^7 + 16 t^8 + 26 t^9 + 44 t^{10} + 
 67 t^{11} + 115 t^{12} + 180 t^{13}$~~} }}& 
 \parbox[t]{5mm}{\multirow{11}{*}{\rotatebox[origin=c]{90}{$2 t^3 + 3 t^4 + 4 t^5 + 6 t^6 + 10 t^7 + 17 t^8 + 28 t^9 + 45 t^{10} + 
 72 t^{11} + 115 t^{12} + 186 t^{13}~~$ }}}
 \\[390pt] \hline
\end{tabular}
\begin{center}
Table 3.    
\end{center}

 \def\arraystretch{3}
 \begin{center}
\begin{tabular}{|c|c|}\hline
~&Final result\\ [0 pt] \hline
  $O_O$ & $\frac{t^4 - 2 t^5 + t^6 + 4 t^7 - 9 t^8 + 4 t^9 + 7 t^{10} - 18 t^{11} + 
   11 t^{12} + 12 t^{13} + 756 t^{14} - 286 t^{15} + 301 t^{16} + 
   474 t^{17}}{(1 - 2 t + t^2 - t^4) (1 - t)}$\\ [2 pt] \hline
 $I_O$ & $\frac{t^2 - t^3 - t^4 + t^5 - 3 t^7 + 2 t^8 - t^9 - 2 t^{10} + 6 t^{11} - 9 t^{12} + 9 t^{13} + 289 t^{14} - 113 t^{15} + 115 t^{16} + 
   180 t^{17}}{(1 - 2 t + t^2 - t^4) (1 - t)}$\\ [2 pt] \hline
$O_I$ & $\frac{t^2 - t^3 - t^4 + t^5 - 3 t^7 + 2 t^8 - t^9 - 2 t^{10} + 6 t^{11} - 9 t^{12} + 9 t^{13} + 289 t^{14} - 113 t^{15} + 115 t^{16} + 
   180 t^{17}}{(1 - 2 t + t^2 - t^4) (1 - t)}$ \\ [2 pt] \hline 
 $I_I$ & $\frac{t^2 - 2 t^3 + t^4 - t^6 + t^{12} + 302 t^{14} - 114 t^{15} + 115 t^{16} + 
   186 t^{17}}{(1 - 2 t + t^2 - t^4) (1 - t)}$
 \\ [2 pt] \hline
\end{tabular}
\end{center}
\begin{center}
    Table 4.
\end{center}
Next we explain the second and third row of table 3 in the following paragraphs:

\subsection{$O_I$ irreducible bridges}
We apply the same procedure as for the $O_O$ case. We examine the complicated bridges first (let us assume that we are on the higher outer line). The difference is that we do not need to add a right step after the walk that goes right on the grid $\mathbb{Z} \times$\{$1, 2$\}, because the irreducible bridge finishes on an inner line, and therefore, the first step of the walk that goes left on the grid $\mathbb{Z} \times$\{$0, 1$\} is a step to the left. Thus, after the rotation by $180$ degrees we will still have a SAW. (The beginning of the rotated walk is not moved.) Hence, after the transformation we have a SAW on the grid $\mathbb{Z} \times$\{$1, 2$\}, which does not have left steps, its starting and ending points are on the same line, its first step is to the right and its length is with $1$ bigger than the length of the walk it was obtained from. 

\subsection{$I_O$ irreducible bridges}
We see that the number of irreducible bridges of types $O_I$ and $I_O$ are equal (Fig. 7 and 8). Hence, we can use the same reasoning.

\begin{figure}[h!]
\begin{tikzpicture}[scale=0.8]
\draw[step=1cm,gray,very thin] (-2,-2) grid (8,1);
\draw[blue,ultra thick,dashed] (-2,1) -- (1,1);
\draw[blue,ultra thick] (1,1) -- (8,1) -- (8,0) -- (6,0) -- (6,-1) -- (4,-1) -- (4,0) -- (1,0) -- (1,-2) -- (2,-2) -- (2,-1) -- (3,-1) -- (3,-2) -- (8,-2) -- (8,-1) ; 

\end{tikzpicture}

\caption{An irreducible bridge of type  $O_I$ with a tail of length $2$. }
\end{figure}

\begin{figure}[h!]
\begin{tikzpicture}[scale=0.8]
\draw[step=1cm,gray,very thin] (-1,-2) grid (9,1);
\draw[blue,ultra thick](-1,1) -- (6,1) -- (6,0) -- (4,0) -- (4,-1) -- (2,-1) -- (2,0) -- (-1,0) -- (-1,-2) -- (0,-2) -- (0,-1) -- (1,-1) -- (1,-2) -- (6,-2) -- (6,-1);
\draw[blue,ultra thick,dashed] (6,-1) -- (9,-1) ; 

\end{tikzpicture}

\caption{Fig. 7 after the transformation to an irreducible bridge of type $I_O$, that moves in the ``opposite'' direction.  }
\end{figure}
\subsection{$I_I$ irreducible bridges}
 Let us assume that the walk starts from the higher inner line. We again perform the same operation with the difference that we do not need to add any right steps because we start and finish the bridge on inner lines and as a result, after the transformation we have a SAW on the grid $\mathbb{Z} \times$\{$1, 2$\} with the same length.

  \section{Lower and upper bounds for the number of SAWs on the $\mathbb{Z}\times\{-1,0,1\}$ and $\mathbb{Z}\times\{-1,0,1,2\}$ lattices.}
  We use the Hammersley Welsh-Method which was used to find bounds for the number of SAWs on the $\mathbb{Z}\times\mathbb{Z}$ lattice (see [MS] p. 57).
  \subsection{Lower bound}
 We introduce Fekete's lemma.
 \begin{lemma}
For every subadditive set $\{a_n\}_{n=1}^{\infty}$, the bound $\lim_{n\rightarrow \infty}\cfrac{a_n}{n}$ exists and is equal to $\inf \cfrac{a_n}{n}$.
\end{lemma}
Let $y$ be equal to $1$ or $2$. Then  
$$c_{[-1,y]n+m} \leq c_{[-1,y]n}c_{[-1,y]m}$$
$$\log c_{[-1,y]n+m} \leq \log c_{[-1,y]n}+\log c_{[-1,y]m}.$$
Therefore,
$$\inf \frac{\log c_{[-1,y]n}}{n}=\lim_{n\rightarrow \infty}\frac{\log c_{[-1,y]n}}{n}$$
$$\log \mu_{[ -1, y]}=\inf \frac{\log c_{[-1,y]n}}{n}$$
$$\mu^n_{[ -1,y]}\leq c_{[-1,y]n} $$for $n\geq1$. 
\subsection{Upper bound}
We will prove (1) (section 2)  for the $\mathbb{Z}\times\{-1,0,1\}$ lattice, however, one can prove (2) (section 3) for the $\mathbb{Z}\times\{-1,0,1,2\}$ lattice in the same way. 
\subsubsection{$\mathbb{Z}\times\{-1,0,1\}$ lattice}
Before going into details, we need several definitions and one lemma.

\begin{definition}
 An $n$-step \textit{half-space walk} (all points after the first one are on the right of the line, parallel to the $x$-axis passing through it) is a SAW, whose $x$-axis coordinates of the points satisfy the following inequality:

$$x_0<x_i $$ for all $ i=1, 2,\dots, n.$

The number of $n$-step half-space walks is denoted with $h_{[-1,1]n}$. By convention, $h_{[-1,1]0}=1$.
\end{definition}

In particular, every bridge is a half-space walk. 

\begin{definition}
\textit{The span} of a $n$-step SAW is the difference between the smallest and biggest $x$-axis coordinate of points, which are part of the walk:
$$\max\limits_{0 \leq i \leq n}x_i-\min\limits_{0 \leq i \leq n}x_i.$$
The number of $n$-step half-space walks (respectively bridges) starting at the origin $(0,0)$ and having span $A$ is denoted by $h_{[-1,1]n,A}$ (respectively $b_{[-1,1]n,A}$). 

\end{definition}

 Note that $h_{[-1,1]n,0}$ is $1$ if $n=0$, and it is $0$ otherwise.
 
 \begin{lemma}
For each integer $A >0 $, let $P_F(A)$ denote the number of partitions of $A$ into distinct integers, whose number is less than $4$ (i.e the number of ways to write $A= A_1+\dots+A_k$, where $A_1>\dots>A_k>0$ and $k\leq 3$). Then 
$$P_F(A) \leq 1+A+A^2.$$

By assumption we let $P_F(0)=1$ since $1 \leq 1 + 0 +0$.
  \end{lemma}

\begin{proof}
If $k=1$, then we have one way. If  $k=2$, then we have at most $A$ ways, as for every choice for $A_1$, we have at most one possibility for  $A_2$. The next case is analogical, as for $k=3$ we count all possibilities for $A_1$ and $A_2$, and we get at most $A^2$ different sums.
\end{proof}
  
The following proposition contains the first part of the proof of our upper bound. 
 
 \begin{proposition}
For every $n\geq0$,
$$h_{[-1,1]n} \leq P_F(n)b_{[-1,1]n}.$$
 \end{proposition}
 
 \begin{proof} 
 
 Let us have an $n$-step half-space walk denoted by  $w$ that starts at the origin  $(0,0)$. Let $n_0=0$. For each $i=1,2,\dots,$ respectively define $A_j(w)$ and $n_j(w)$ so that
 
 $$ A_i=\max\limits_{n_{i-1} \leq j \leq n}(-1)^i(x_{n_{i-1}}-x_j)$$

and $n_i$ is the largest value of $j$, for which this maximum is attained. The recursion is stopped at the smallest integer $k$ such that  $n_k=n$; this means that $A_{k+1}(w)$ and $n_{k+1}(w)$ are not defined. Observe that $A_1(w)$ is the span of $w$; in general $A_{i+1}$ is the span of the SAW $(w_{n_i},\dots, w_n)$ ($w_i$ is the $i+1$ point of the walk  $w$), which is either a half-space walk or the reflection of one. Moreover, each of the subwalks $(w_{n_i},\dots, w_{n_{i+1}} )$ is either a bridge or the reflection of one. Also observe that $A_1>A_2>\dots>A_k>0$ and that as we are working on the grid $\mathbb{Z} \times$\{$-1, 0, 1,$\}, then $k\leq 3$, because when we reach $w_{n_1}$, the rest of the walk cannot have points on at least one of the lines.  This follows from the fact that the last rightmost point is higher or lower (let us assume that it is lower) from the first rightmost point, and all other points need to have a smaller $x$-axis coordinate. As a result, since we have a half-space walk, all the points need to be below the walk with ends $(w_{n_0},w_{n_1})$. From here we can see that the statement is true.  As the same process continues, we find that $k \leq 3$.  

\begin{figure}[h!]
\begin{tikzpicture}[scale=0.8]

\node[draw] at (-3,1) {$w_0$};
\node[draw] at (3,-1) {$w_{n_1}$};
\node[draw] at (-2,-2) {$w_{n_2}$};
\node[draw] at (2,-2) {$w_{n_3}$};
\draw[blue,ultra thick] (-2,1) -- (1,1) -- (1,0) -- (2,0) -- (2,-1);
\draw[blue,ultra thick,dashed](2,-1) -- (0,-1) -- (0,0) -- (-1,0) -- (-1,-2);
\draw[blue,ultra thick,dotted] (-1,-2) -- (1,-2) ;
 
\end{tikzpicture}

\caption{A half-space walk $w$ in $H_{14}[4,3,2]$. }
\end{figure}
 
 \begin{figure}[h!]
\begin{tikzpicture}[scale=0.8]

\draw[blue,ultra thick] (-2,1) -- (1,1) -- (1,0) -- (2,0) -- (2,-1);
\draw[blue,ultra thick,dashed](2,-1) -- (4,-1) -- (4,0) -- (5,0) -- (5,-2);
\draw[blue,ultra thick,dotted] (5,-2) -- (3,-2) ;
 
\end{tikzpicture}

\caption{The transformed walk $w\textprime$ in $H_{14}[7,2]$.}
\end{figure}

For every decreasing sequence of $k$ positive integers $a_1>a_2>\dots>a_k>0,$ let $H_n[a_1,a_2,\dots,a_k]$ be the set of $n$-step half-space walks $w$ with  $w_0=(0,0)$ and $A_1(w)=a_1$, \dots, $A_k(w)=a_k$ and $n_k(w)=n$. Note that in particular $H_n[a]$ is the set of $n$-step bridges of span $a$.
 
Given an $n$-step half-space walk $w$, define a new $n$-step walk $w\textprime$ as follows: for $0 \leq j \leq n_1(w)$, define $w\textprime_j=w_j$; and for $n_1(w)<j \leq n$, define $w\textprime_j$ to be the reflection of the point $w_j$ in the hyperplane \textit{${x_1}$}$=A_1(w)$. Observe that if $w$ is in $H_n[a_1, a_2, \dots, a_k]$, then $w\textprime$ is in $H_n[a_1+a_1, a_3, \dots, a_k]$; moreover, this transformation maps an unique walk (the transformation is one-to-one), so

  $$|H_n[a_1, a_2, \dots, a_k]|\leq|H_n[a_1+a_2,a_3,\dots,a_k]|.$$
  
  Therefore, summing over all possible sequences $a_1>\dots>a_k>0$, we get that
  
  $$h_{[-1,1]n}=\sum|H_n[a_1,\dots, a_k]| \leq \sum|H_n[a_1+\dots+a_k]|=$$
  $$=\sum b_{[-1,1]n,a_1+\dots+a_k}$$  
which tells us that 
  
  $$h_{[-1,1]n} \leq \sum P_F(A)b_{[-1,1]n,A}.$$
  
  Note that $P_F(A)\leq P_F(n)$ for $A \leq n$. Hence,
  
  $$h_{[-1,1]n} \leq P_F(n) \sum_{A=1}^{n} b_{[-1,1]n,A}=P_F(n)b_{[-1,1]n}$$
  
 which proves the proposition.
   \end{proof}

 Therefore, we are ready to prove our upper bound.

  \begin{proof}
Given an arbitrary $n$-step SAW $w$, let $M=\min x_j$ and $m$ be the largest $j$ such that $x_j=M$. Then $(w_m,... w_n)$ is a half-space walk, as is
$$(w(m)-(1,0),w(m),w(m-1),\dots,w(0))$$

Using this decomposition method, as well as proposition A.10 and the inequality $b_{[-1,1]i}b_{[-1,1]j} \leq b_{[-1,1]i+j}$ (follows from the fact that whenever we concatenate two bridges we get a new bridge) we obtain:

$$c_{[-1,1]n} \leq \sum_{m=0}^{n} h_{[-1,1]n-m}h_{[-1,1]m+1} \leq \sum_{m=0}^{n} b_{[-1,1]n-m}b_{[-1,1]m+1}P_F(m+1)P_F(n-m) $$

$$\leq b_{[-1,1]n+1}\sum_{m=0}^{n} (1+(m+1)+(m+1)^2)(1+(n-m)+(n-m)^2)$$

$$\leq b_{[-1,1]n+1}((n+1)+2(n+1)^2+3(n+1)^3+2(n+1)^4+(n+1)^5) $$

for all $n$. Therefore, since $b_{[-1,1];n+1} \leq \mu^{n+1}_{[-1,1]}$ (the proof of this fact, as for our lower bound for the number of SAWs, uses Fekete's lemma, when concerning superadditive sequences), we have 
$$c_{[-1,1]n}\leq \mu^{n+1}_{[-1,1]}((n+1)+2(n+1)^2+3(n+1)^3+2(n+1)^4+(n+1)^5).$$\end{proof}
\begin{corollary}

For all $n \geq 2$ we have that:
$$\frac{\mu_{[-1,1]}^{n-1}}{(n+1)+2(n+1)^2+3(n+1)^3+2(n+1)^4+(n+1)^5}  \leq b_{[-1,1]n} \leq \mu_{[-1,1]}^{n}.$$

\end{corollary}  
 \begin{proof} The left bound follows from (1) and from the fact that $\mu^{n}_{[-1,1]} \leq c_{[-1,1]n}$. The right bound can be proven by using the superadditive sequence $\{\log b_{[-1,1]n}\}_{n=1}^{\infty}$. Thus, $\lim_{n\rightarrow\infty}\sqrt[n]{b_{[-1,1]n}}=\mu_{[-1,1]}$.
  \end{proof}
\subsubsection{$\mathbb{Z}\times\{-1,0,1,2\}$ lattice}
When using the method for the $\mathbb{Z}\times \{-1,0,1,2\}$ lattice we would have that $P_F(A)\leq 1+A+A^2+A^3$. Hence, we would have that
$$c_{[-1,2]n} \leq \sum_{m=0}^{n} h_{[-1,2]n-m}h_{[-1,2]m+1} \leq \sum_{m=0}^{n} b_{[-1,2]n-m}b_{[-1,2]m+1}P_F(m+1)P_F(n-m) $$

$$\leq b_{[-1,2]n+1}\sum_{m=0}^{n} (1+(m+1)+(m+1)^2+(m+1)^3)(1+(n-m)+(n-m)^2+(n-m)^3)$$

$$\leq b_{[-1,2]n+1}((n+1)+2(n+1)^2+3(n+1)^3+4(n+1)^4+3(n+1)^5+2(n+1)^6+(n+1)^7) $$

for all $n\geq1$. Thus, since $b_{[-1,2]n+1} \leq  \mu_{[-1,2]}^{n+1}$ (the proof again uses Fekete's lemma), we have
$$c_{[-1,2]n}\leq \mu^{n+1}_{[-1,2]}((n+1)+2(n+1)^2+3(n+1)^3+4(n+1)^4+3(n+1)^5+2(n+1)^6+(n+1)^7).$$
We have the analogous corollary that
$$\lim_{n\rightarrow\infty}\sqrt[n]{b_{[-1,2]n}}=\mu_{[-1,2]}.$$
\end{document}